\documentclass[10pt,reqno]{amsart}

\usepackage[T1]{fontenc}
\usepackage{lmodern}
\usepackage{microtype}
\usepackage{amsmath,amssymb,mathrsfs,mathtools}
\usepackage{booktabs}
\usepackage{enumitem}
\usepackage{needspace}
\usepackage{xcolor}
\usepackage[colorlinks=true,linkcolor=blue!55!black,
  citecolor=blue!55!black,urlcolor=blue!55!black]{hyperref}
\usepackage{cleveref}
\usepackage{aliascnt}

\allowdisplaybreaks[3]
\setlist{nosep,leftmargin=2em}

\newtheorem{theorem}{Theorem}[section]
\newaliascnt{proposition}{theorem}
\newtheorem{proposition}[proposition]{Proposition}
\aliascntresetthe{proposition}
\newaliascnt{lemma}{theorem}
\newtheorem{lemma}[lemma]{Lemma}
\aliascntresetthe{lemma}
\newaliascnt{corollary}{theorem}
\newtheorem{corollary}[corollary]{Corollary}
\aliascntresetthe{corollary}

\theoremstyle{definition}

\crefname{proposition}{Proposition}{Propositions}
\crefname{theorem}{Theorem}{Theorems}
\crefname{lemma}{Lemma}{Lemmas}
\crefname{corollary}{Corollary}{Corollaries}

\newcommand{\Q}{\mathbb Q}
\newcommand{\Z}{\mathbb Z}
\newcommand{\E}{\mathbb E}
\newcommand{\Prob}{\mathbb P}
\newcommand{\F}{\mathbb F}

\newcommand{\R}{\mathbb R}

\newcommand{\Gal}{\operatorname{Gal}}

\newcommand{\rad}{\operatorname{rad}}

\newcommand{\ord}{\operatorname{ord}}
\newcommand{\Gcal}{G}
\newcommand{\Pcal}{P}
\newcommand{\Li}{\operatorname{Li}}
\newcommand{\eps}{\varepsilon}
\newcommand{\Nm}{\operatorname N}

\title[Density of Vanishing of Certain Eigenspaces of Cyclotomic Class Groups]
{Density of Vanishing of Certain Eigenspaces of Cyclotomic Class Groups}

\author[X. Guo]{Xuejun Guo}
\address{School of Mathematics, Nanjing University, Nanjing 210093,
People's Republic of China}
\email{guoxj@nju.edu.cn}

\author[Z. Tao]{Zhengyu Tao}
\address{School of Mathematics, Hefei University of Technology,
Hefei 230009, People's Republic of China}
\email{taozhy@hfut.edu.cn}
\date{}

\hypersetup{
  pdftitle={Density of Vanishing of Certain Eigenspaces of Cyclotomic Class Groups},
  pdfauthor={Xuejun Guo and Zhengyu Tao}
}

\subjclass[2020]{Primary 11R18, 11M06; Secondary 11L40, 11R23}
\keywords{Vandiver's conjecture, cyclotomic fields, Dirichlet $L$-functions,
random Euler products, generalized Bernoulli numbers, reflection theorem}

\begin{document}

\begin{abstract}
For an odd prime $p$, let $A_j$ be the $\omega^j$-eigenspace of the
$p$-primary class group of $\mathbb Q(\zeta_p)$.
Fix an even integer $d\ge4$, put $N=(p-1)/d$, and let
$U_d=(\mathbb Z/d\mathbb Z)^\times$.
For a relative density-one set of primes $p\equiv d+1\pmod{2d}$,
we prove that the odd block $\bigoplus_{a\in U_d}A_{aN}$ has order
at most $p^{\varphi(d)/2-1}$, and reflection shows that the even block
$\bigoplus_{a\in U_d}A_{p-aN}$ has $p$-rank at most
$\varphi(d)/2-1$.  For each \(d\in\{4,6\}\), both blocks vanish for a relative density-one set of primes \(p\equiv d+1\pmod{2d}\); in particular, the even components vanish in accordance with Vandiver's conjecture.
For each \(d\in\{8,10,12\}\), the odd block has order at most \(p\) for a relative density-one set of primes in the same progression, so every summand \(A_{aN}\), \(a\in U_d\), is cyclic, in accordance with Iwasawa's cyclicity conjecture. The proof uses an atomless limiting law for products of Dirichlet
$L$-values, integrality of generalized Bernoulli norms, the relative
class-number formula and reflection.
A separate exact computation proves $A_{34}=0$ for every odd
prime $p$; a single-file PARI/GP program reproduces the calculation.
\end{abstract}

\maketitle

\section{Introduction}

Let $p$ be an odd prime, and let $A$ be the $p$-primary class
group of $K_p=\Q(\zeta_p)$.
For $\sigma_a(\zeta_p)=\zeta_p^a$, let
$\omega(\sigma_a)\in\Z_p^\times$ be the Teichm\"uller lift of
$a\in(\Z/p\Z)^\times$.
Put
\[
 e_j=\frac1{p-1}\sum_{a=1}^{p-1}\omega(a)^{-j}\sigma_a,
 \qquad A_j=e_jA.
\]
Indices are taken modulo $p-1$.
If $A^+$ is the $p$-primary class group of the maximal real
subfield $K_p^+$, extension of ideals gives
\[
 A=\bigoplus_{j=0}^{p-2}A_j,\qquad
 A^+\simeq
 \bigoplus_{\substack{0\le j\le p-2\\j\text{ even}}}A_j.
\]
Vandiver's conjecture is the assertion
\begin{equation}\label{eq:vandiver-eigenspaces}
 A_j=0\qquad(j\text{ even}).
\end{equation}
Iwasawa's cyclicity conjecture asserts that $A_j$ is cyclic
for every odd $j$; see \cite[p.~223]{Kurihara}.

Hart, Harvey and Ong \cite{HartHarveyOng} verified Vandiver's conjecture
for all primes $p<2^{31}$.
Several individual components are known to vanish.
Reflection gives $A_{(p+1)/2}=0$ for $p\equiv3\pmod4$;
see also \cite{Osburn,QS}.
Kurihara \cite{Kurihara} proved $A_{p-3}=0$.
The vanishing of $K_8(\Z)$ proved in \cite{DutourSikiric2019}
implies $A_{p-5}=0$.
More generally, Soul\'e \cite[Section~3.2]{SoulePerfectForms}
proved that, for every odd integer $n\ge5$,
\[
 A_{p-n}=0\qquad\text{whenever}\qquad
 \log p>n^{224n^4}.
\]
His argument uses algebraic $K$-theory and Voronoi's reduction theory
of positive definite quadratic forms.
Ribet \cite{Ribet1976} proved the converse to Herbrand's theorem
in 1976, completing the criterion
\[
  A_{p-2k}\ne0
  \quad\Longleftrightarrow\quad
  p\mid\operatorname{num}(B_{2k})
  \qquad (2\le2k\le p-3).
\]
The Herbrand direction, together with the reflection theorem,
implies that $A_{2k}=0$ whenever
$p\nmid\operatorname{num}(B_{2k})$.
For fixed even indices, Zhang and Sun \cite{ZS} proved
$A_{2i}=0$ for $1\le i\le14$, and Qin \cite{Q} extended
this range to $1\le i\le16$.
Chen et al.~\cite{ChenEtAl} proved that, for each $\alpha>1/2$,
a density-one set of primes satisfies
\[
  A_{p-2k}=A_{2k}=0
  \qquad\left(2\le2k\le\frac{\sqrt p}{(\log p)^\alpha}\right).
\]

In the function-field setting, Taelman
\cite{TaelmanHerbrandRibet} established a Herbrand--Ribet
theorem for the Carlitz module in 2012.
However, the corresponding Kummer--Vandiver analogue is false:
Angl\`es and Taelman \cite{AnglesTaelmanVandiver} gave explicit
counterexamples using Artin--Schreier base change in 
2013.
For an even integer $d\ge4$, define
\[
 \Pcal_d=\{p\text{ prime}:p\equiv d+1\pmod{2d}\},\qquad
 \pi_d(X)=\#\{p\le X:p\in\Pcal_d\},
\]
and represent $(\Z/d\Z)^\times$ by
\[
 U_d=\{a:1\le a<d,\ \gcd(a,d)=1\}.
\]
For $p\in\Pcal_d$, $N=(p-1)/d$ is odd.
A set $S_d\subseteq\Pcal_d$ has relative density one if
$\#\{p\le X:p\in S_d\}/\pi_d(X)\to1$.
For a finite abelian $p$-group $M$, put
$r_p(M)=\dim_{\F_p}(M/pM)$.

\begin{theorem}\label{th1.1}
Fix an even integer $d\ge4$.
There is a set $S_d\subseteq\Pcal_d$ of relative density one
such that, for $p\in S_d$ and $N=(p-1)/d$,
\begin{equation}\label{theorem1.1}
 \left|\bigoplus_{a\in U_d}A_{aN}\right|
 \le p^{\varphi(d)/2-1}.
\end{equation}
In particular, by the reflection theorem, for every $p\in S_d$,
\[
 r_p\!\left(\bigoplus_{a\in U_d}A_{p-aN}\right)
 \le\frac{\varphi(d)}2-1.
\]
\end{theorem}

\begin{corollary}\label{cor1.2}
Let $N=(p-1)/d$.
\begin{enumerate}
\item For $d=4,6$, a relative density-one set of primes
$p\in\Pcal_d$ satisfies
\[
 \bigoplus_{a\in U_d}
       \bigl(A_{aN}\oplus A_{p-aN}\bigr)=0.
\]
\item For $d=8,10,12$, a relative density-one set of primes
$p\in\Pcal_d$ has odd block
$\bigoplus_{a\in U_d}A_{aN}$ trivial or isomorphic to $\Z/p\Z$.
\end{enumerate}
\end{corollary}

For $d=4,6$, Thaine  \cite{Thaine1995,Thaine2002}  discussed the cases
$p\equiv 5\pmod 8$ and $p\equiv 7\pmod{12}$ by Gaussian
periods. In our notation, Thaine's idea is to show that, if \(A_{p-aN}\) is nontrivial, then every prime \(q\) satisfying \(\operatorname{ord}_p(q)=N\) must have a certain special form, with the hope of deriving a contradiction from a suitable version of Dirichlet's theorem on primes in arithmetic progressions.
Although Thaine showed that, in these two cases, the nontriviality of \(A_{(p+3)/4}\) or \(A_{(5p+1)/6}\) forces \(p\) to have a special form, this observation has not yet been successfully combined with any density theorem to yield a contradiction.

The proof rests on a divisibility obstruction supplied by a limiting
distribution.
For $p\in\Pcal_d$, $p>d+1$, and a character $\psi_p$ of exact
order $d$, define
\[
 G_{d,p}=\Nm_{\Q(\zeta_d)/\Q}(B_{1,\psi_p^{-1}}),
 \qquad m_d=\varphi(d)/2.
\]
We prove that $G_{d,p}$ is a positive integer and that
\[
 \left|\bigoplus_{a\in U_d}A_{aN}\right|=p^{v_p(G_{d,p})},
 \qquad
 \frac{G_{d,p}}{p^{m_d}}
 =\frac1{\pi^{\varphi(d)}}\prod_{a\in U_d}L(1,\psi_p^a).
\]
The product of $L$-values has an atomless limiting law.
Hence the normalized values $G_{d,p}/p^{m_d}$ lie in the
closed set $\Z_{>0}$ with relative frequency tending to zero.
Equivalently, $v_p(G_{d,p})\le m_d-1$ on a relative density-one set.
This gives the odd-block bound; reflection gives the even-block
rank bound.

As an independent computational result, we obtain the following fixed-index result, whose proof
is given in Appendix~\ref{app:A34-computation}.
\begin{theorem}\label{thm:A34-vanishing}
For every odd prime $p$, $A_{34}=0$.
\end{theorem}
The factorization
\[
 B_{34}=\frac{17\cdot151628697551}{6}
\]
and reflection reduce the proof to the prime $151628697551$.
At this prime, a cyclotomic-unit test gives a nonidentity residue
by exact modular arithmetic.
The accompanying PARI/GP program \texttt{verify\_A34.gp}
evaluates the complete product used in this computation.

Section~\ref{sec:local} establishes the local laws and constructs
the atomless random Euler product.
Section~\ref{sec:weak} proves the limiting law, identifies the
Bernoulli norm with the odd-block order, and deduces
\cref{th1.1,cor1.2}.
Appendix~\ref{app:A34-computation} gives the fixed-index computation.

\section{The random Euler product}\label{sec:local}

Fix an even integer $d\ge4$, and put
$K=\Q(\zeta_d)$ and $\mu_d=\{\zeta_d^j:0\le j<d\}$, where
$\zeta_d=e^{2\pi i/d}$.
Every $p\in\Pcal_d$ splits completely in $K$.
For a prime ideal $\mathfrak p\mid p$, reduction identifies $\mu_d$
with the subgroup of order $d$ in $\F_p^\times$.
Define the character $\chi_{\mathfrak p}$ modulo $p$ by
\[
 \chi_{\mathfrak p}(a)=\left(\frac a{\mathfrak p}\right)_d,\qquad
 \chi_{\mathfrak p}(a)\equiv a^{(p-1)/d}\pmod{\mathfrak p}
 \quad(p\nmid a),
\]
and set $\chi_{\mathfrak p}(a)=0$ for $p\mid a$.
It has exact order $d$ and satisfies
$\chi_{\mathfrak p}(-1)=(-1)^{(p-1)/d}=-1$; hence it is odd and
primitive.

For $c\in\Q^\times$, choose $\alpha\in\overline{\Q}$
with $\alpha^d=c$, and let $\eta_{d,c}$ be the finite-order Hecke
character corresponding to the Kummer character
\[
 \Gal(K(\alpha)/K)\longrightarrow\mu_d,
 \qquad \sigma\longmapsto\frac{\sigma(\alpha)}{\alpha}.
\]
It is independent of the choice of $\alpha$ and depends only on the
class of $c$ modulo $\Q^{\times d}$.
For $p\in\Pcal_d$ with $v_p(c)=0$ and every $\mathfrak p\mid p$,
we have
$
 \eta_{d,c}(\mathfrak p)=\chi_{\mathfrak p}(c).
$
For a Hecke character $\vartheta$, write $\mathfrak f(\vartheta)$ for
its finite conductor ideal and 
$\Nm(\mathfrak f(\vartheta))=
 |\mathcal O_K/\mathfrak f(\vartheta)|$ for its absolute norm.
Throughout this section, let
\[
 S=\{\ell\text{ prime}:\ell\mid d\}.
\]

\begin{lemma}\label{lem:conductor}
For every $c\in\Q^\times$,
\begin{equation}\label{eq:conductor}
 \Nm(\mathfrak f(\eta_{d,c}))
 \le \bigl(d^2\rad(d)\bigr)^{\varphi(d)}
 \prod_{\substack{\ell\notin S\\ d\nmid v_\ell(c)}}\ell^{\varphi(d)}.
\end{equation}
Here $\rad(d)=\prod_{\ell\mid d}\ell$.
For $\ell\notin S$, the local conductor exponent at every
$\mathfrak l\mid\ell$ is $1$ if $d\nmid v_\ell(c)$ and $0$ otherwise.
\end{lemma}

\begin{proof}
For each prime ideal $\mathfrak l$ of $K$ above
a rational prime $\ell$, write $e_{\mathfrak l}=e(\mathfrak l/\ell)$,
$f_{\mathfrak l}=f(\mathfrak l/\ell)$, and
$a_{\mathfrak l}=a_{\mathfrak l}(\eta_{d,c})$.  Then
\[
 \Nm(\mathfrak f(\eta_{d,c}))
 =\prod_\ell\ell^{\sum_{\mathfrak l\mid\ell}
                                  f_{\mathfrak l}a_{\mathfrak l}}.
\]
Fix $\mathfrak l\mid\ell$, let $F=K_{\mathfrak l}$, and let
$\mathfrak m$ be the maximal ideal of $\mathcal O_F$.
Normalize the valuation $v$ on $F$ by $v(F^\times)=\mathbb Z$.
Choose $\alpha$ with $\alpha^d=c$ and put $E=F(\alpha)$.
Since $\mu_d\subset F$, the extension $E/F$ is cyclic of degree
dividing $d$, and the local Galois character defining $\eta_{d,c}$ is
the faithful character
\[
 \operatorname{Gal}(E/F)\longrightarrow\mu_d,
 \qquad \sigma\longmapsto\sigma(\alpha)/\alpha.
\]
Let $\theta_{\mathfrak l}:F^\times\to\mu_d$ be the corresponding
character under local reciprocity.  Its conductor exponent
$a_{\mathfrak l}$ is the least integer $n\ge0$ for which it is
trivial on $U_F^n$, where $U_F^0=\mathcal O_F^\times$ and
$U_F^n=1+\mathfrak m^n$ for $n\ge1$.

Suppose first that $\ell\notin S$.
Then $\ell\nmid d$, and $K/\Q$ is unramified at $\ell$, so
$e_{\mathfrak l}=1$ and $v(c)=v_\ell(c)$.
If $d\nmid v_\ell(c)$, the extension of $v$ to $E$ satisfies
\[
 v(\alpha)=\frac{v_\ell(c)}d\notin\mathbb Z.
\]
An unramified extension of $F$ has the same value group $\mathbb Z$,
so $E/F$ is ramified.  Its degree divides $d$ and is therefore
prime to $\ell$; hence $E/F$ is tamely ramified.
The faithful character above is nontrivial on inertia and trivial
on wild inertia.  Equivalently, $\theta_{\mathfrak l}$ is nontrivial
on $U_F^0$ and trivial on $U_F^1$, so $a_{\mathfrak l}=1$.

If $d\mid v_\ell(c)$, write $c=\pi^{dm}u$ with $m\in\mathbb Z$,
a uniformizer $\pi$ of $F$, and $u\in\mathcal O_F^\times$.
Then $E=F(\alpha/\pi^m)$ and $(\alpha/\pi^m)^d=u$.
Choose a finite extension $k'$ of the residue field of $F$ in which
$X^d-\bar u$ has a root, and let $F'/F$ be the unramified extension
with residue field $k'$.  This root is nonzero and simple, since
$\ell\nmid d$, so Hensel's lemma lifts it to a root $\beta\in F'$
of $X^d-u$.  The ratio $(\alpha/\pi^m)/\beta$ belongs to
$\mu_d\subset F$.
Thus $E=F(\beta)\subseteq F'$, so $E/F$ is unramified and
$a_{\mathfrak l}=0$.
These conclusions hold for every $\mathfrak l\mid\ell$.  As
$\sum_{\mathfrak l\mid\ell}f_{\mathfrak l}=[K:\Q]=\varphi(d)$
at such an unramified prime, we have
\[
 \sum_{\mathfrak l\mid\ell}f_{\mathfrak l}a_{\mathfrak l}
 =\begin{cases}
   \varphi(d),&d\nmid v_\ell(c),\\
   0,&d\mid v_\ell(c).
  \end{cases}
\]

Now suppose that $\ell\in S$, so $\ell\mid d$.
Put $t=v(d)=e_{\mathfrak l}v_\ell(d)$.
For any $u\in1+\mathfrak m^{2t+1}$, the polynomial $g(X)=X^d-u$
satisfies
\[
 v(g(1))=v(1-u)\ge2t+1>2t=2v(g'(1)).
\]
Hensel's lemma therefore gives a root of $g$ in $F$, proving
$1+\mathfrak m^{2t+1}\subseteq F^{\times d}$.
For every $x\in F^\times$ we have
$\theta_{\mathfrak l}(x^d)=\theta_{\mathfrak l}(x)^d=1$,
so $\theta_{\mathfrak l}$ is trivial on $1+\mathfrak m^{2t+1}$.
Consequently,
\[
 a_{\mathfrak l}\le2t+1=2e_{\mathfrak l}v_\ell(d)+1.
\]
Using $\sum_{\mathfrak l\mid\ell}e_{\mathfrak l}f_{\mathfrak l}
=\varphi(d)$ and $e_{\mathfrak l}\ge1$, we obtain
\[
 \sum_{\mathfrak l\mid\ell}f_{\mathfrak l}
       a_{\mathfrak l}
 \le 2v_\ell(d)\sum_{\mathfrak l\mid\ell}
             e_{\mathfrak l}f_{\mathfrak l}
       +\sum_{\mathfrak l\mid\ell}f_{\mathfrak l}
 \le \varphi(d)(2v_\ell(d)+1).
\]
Thus the contribution to the conductor norm from primes above
$S$ is at most
\[
 \prod_{\ell\mid d}\ell^{\varphi(d)(2v_\ell(d)+1)}
 =\bigl(d^2\rad(d)\bigr)^{\varphi(d)}.
\]
Multiplying this bound by the contributions from
$\ell\notin S$ proves \eqref{eq:conductor}.
\end{proof}

Let $T\supseteq S$ be a fixed finite set of rational primes, and put
$\pi_{d,T}(X)=\#\{p\le X:p\in\Pcal_d,\ p\notin T\}$.
Define the empirical probability measure on $\mu_d^T$ by
\begin{equation}\label{eq:empirical-local}
 \nu_{d,X,T}=
 \frac1{\varphi(d)\pi_{d,T}(X)}
 \sum_{\substack{p\le X,\ p\in\Pcal_d\\p\notin T}}
 \ \sum_{\mathfrak p\mid p}
 \delta_{\bigl((\ell/\mathfrak p)_d\bigr)_{\ell\in T}}.
\end{equation}
Here $\pi_{d,T}(X)=\pi_d(X)+O_T(1)$.

\begin{proposition}
\label{prop:local-law}
Let $u_d$ be the uniform probability measure on $\mu_d$.
There is a probability measure $\nu_{d,S}$ on
$\mu_d^S$ such that, for every fixed finite set of rational primes
$T\supseteq S$, the measures $\nu_{d,X,T}$ converge, as $X\to\infty$,
to $\nu_{d,S}\otimes u_d^{\otimes(T\setminus S)}$.
Explicitly, for every $\boldsymbol z\in\mu_d^T$,
\begin{equation}\label{eq:local-pointwise}
 \lim_{X\to\infty}\nu_{d,X,T}(\{\boldsymbol z\})
 =\nu_{d,S}(\{\boldsymbol z|_{S}\})\,d^{-|T\setminus S|}.
\end{equation}
\end{proposition}

\begin{proof}
Fix a finite set $T\supseteq S$.
For $\boldsymbol e\in\{0,\ldots,d-1\}^T$, put
$c(\boldsymbol e)=\prod_{\ell\in T}\ell^{e_\ell}$.
For any measure $\nu$ on $\mu_d^T$, use the Fourier convention
\[
 \widehat\nu(\boldsymbol e)
 =\sum_{\boldsymbol z\in\mu_d^T}
      \nu(\{\boldsymbol z\})\prod_{\ell\in T}z_\ell^{e_\ell}.
\]
By Fourier inversion, it suffices to compute the limits of the
Fourier coefficients of $\nu_{d,X,T}$.
If $p\in\Pcal_d$ and $p\notin T$, multiplicativity of the residue symbol gives
$\prod_{\ell\in T}(\ell/\mathfrak p)_d^{e_\ell}
=\eta_{d,c(\boldsymbol e)}(\mathfrak p)$ for every $\mathfrak p\mid p$.
Hence \eqref{eq:empirical-local} gives
\begin{equation}
 \widehat\nu_{d,X,T}(\boldsymbol e)
 =\frac1{\varphi(d)\pi_{d,T}(X)}
   \sum_{\substack{p\le X,\ p\in\Pcal_d\\p\notin T}}
   \sum_{\mathfrak p\mid p}\eta_{d,c(\boldsymbol e)}(\mathfrak p).
 \label{eq:local-Fourier}
\end{equation}
Fix $\boldsymbol e$ and abbreviate $\eta=\eta_{d,c(\boldsymbol e)}$.
For each Dirichlet character $\xi$ modulo $2d$, let $\rho_\xi$
denote the primitive Hecke character associated to
\[
 \mathfrak a\longmapsto\xi(\Nm\mathfrak a)
 \qquad ((\mathfrak a,2d)=1).
\]
Its conductor divides $(2d)\mathcal O_K$.  Indeed, for
$a\in\mathcal O_K$ with $a\equiv1\pmod{(2d)\mathcal O_K}$,
each Galois conjugate of $a$ has the same congruence, so
$N_{K/\Q}(a)\equiv1\pmod{2d}$.  Since $K$ is totally imaginary,
$N_{K/\Q}(a)>0$, and hence
$\xi(\Nm((a)))=\xi(N_{K/\Q}(a))=1$.
Such principal ideals generate the principal ray subgroup for
$(2d)\mathcal O_K$, so the norm character factors through that
ray class group.
Thus $\rho_\xi$ is unramified at every prime above $\ell\notin S$.
For every positive integer $n$, Dirichlet character orthogonality gives
\begin{equation}\label{eq:selector}
 \mathbf 1_{n\equiv d+1\ (2d)}
 =\frac1{\varphi(2d)}
   \sum_{\xi\,(\mathrm{mod}\,2d)}\overline{\xi(d+1)}\xi(n).
\end{equation}
Dirichlet characters are extended by zero on integers not coprime to $2d$.

For $p\nmid d$, the residue degree of a prime of $K$ above $p$ is
the multiplicative order of $p$ modulo $d$.  Thus a prime ideal
$\mathfrak p\mid p$ has residue degree one exactly when
$p\equiv1\pmod d$; in that case $p$ splits into $\varphi(d)$
prime ideals, each of norm $p$.
Since $p\equiv d+1\pmod{2d}$ implies $p\equiv1\pmod d$,
the numerator in \eqref{eq:local-Fourier} equals
\[
 \sum_{\substack{\Nm\mathfrak p\le X,\ f(\mathfrak p/p)=1\\
                  \mathfrak p\mid p,\ p\notin T}}
       \eta(\mathfrak p)\,
       \mathbf 1_{\Nm\mathfrak p\equiv d+1\ (2d)}.
\]
Insert \eqref{eq:selector} into this sum.  For $p\notin T$, both
$\eta$ and $\rho_\xi$ are unramified at $\mathfrak p\mid p$, and
$\xi(\Nm\mathfrak p)=\rho_\xi(\mathfrak p)$.
We may then sum over all prime ideals of norm at most $X$:
those of residue degree at least two lie above rational primes
$p\le X^{1/2}$, so there are $O_d(X^{1/2})$ of them;
the prime ideals above $T$ number $O_{d,T}(1)$.
In the resulting unrestricted sums, each product $\eta\rho_\xi$
is evaluated as its associated primitive Hecke character, extended
by zero at its conductor.  This convention can change the preceding
sum only at primes above $T$.
Consequently,
\[
 \widehat\nu_{d,X,T}(\boldsymbol e)=
 \frac1{\varphi(d)\varphi(2d)\pi_{d,T}(X)}
 \sum_{\xi\,(\mathrm{mod}\,2d)}\overline{\xi(d+1)}
 \sum_{\Nm\mathfrak p\le X}
        (\eta\rho_\xi)(\mathfrak p)
 +O_{d,T}\!\left(\frac{X^{1/2}+1}{\pi_{d,T}(X)}\right).
\]
The prime number theorem in arithmetic progressions, with the
finitely many primes in $T$ omitted, gives
\[
 \pi_{d,T}(X)\sim\frac{\Li(X)}{\varphi(2d)}.
\]
For each fixed finite-order Hecke character $\vartheta$ of $K$,
the prime-ideal theorem gives
\[
 \sum_{\Nm\mathfrak p\le X}\vartheta(\mathfrak p)
 =\begin{cases}
   \Li(X)+o(\Li(X)),&\vartheta=1,\\
   o(\Li(X)),&\vartheta\ne1.
  \end{cases}
\]
The characters $\eta\rho_\xi$ are fixed as $X\to\infty$, since
$d$, $T$, and $\boldsymbol e$ are fixed.  Applying this formula to
each of them yields
\begin{equation}\label{eq:local-Fourier-limit}
 \lim_{X\to\infty}\widehat\nu_{d,X,T}(\boldsymbol e)
 =\frac1{\varphi(d)}
 \sum_{\substack{\xi\,(\mathrm{mod}\,2d)\\
       \eta_{d,c(\boldsymbol e)}\rho_\xi=1}}
 \overline{\xi(d+1)}.
\end{equation}
Taking $T=S$ and applying Fourier inversion defines a probability
measure $\nu_{d,S}$ on $\mu_d^S$ by
\[
 \nu_{d,S}(\{\boldsymbol z\})
 :=\lim_{X\to\infty}\nu_{d,X,S}(\{\boldsymbol z\})
 \qquad(\boldsymbol z\in\mu_d^S).
\]

For any $c\in\Q^\times$, if $\ell\notin S$ and $d\nmid v_\ell(c)$,
then $\eta_{d,c}$ is ramified above $\ell$ by
\cref{lem:conductor}, whereas $\rho_\xi$ is unramified there.
On the inertia group at a prime above $\ell$, multiplication by
the unramified character $\rho_\xi$ does not change the nontrivial
restriction of $\eta_{d,c}$.  Hence $\eta_{d,c}\rho_\xi$ is ramified
there and cannot be the trivial character.
In particular,
\begin{equation}\label{eq:principal-support}
 \eta_{d,c}\rho_\xi\text{ is the trivial Hecke character}
 \quad\Longrightarrow\quad
 d\mid v_\ell(c)\quad(\ell\notin S).
\end{equation}
Now let $T\supseteq S$ be arbitrary.
If $e_\ell\ne0$ for some $\ell\in T\setminus S$, then
$v_\ell(c(\boldsymbol e))=e_\ell\in\{1,\ldots,d-1\}$.
By \eqref{eq:principal-support}, every $\eta_{d,c(\boldsymbol e)}\rho_\xi$
is nontrivial.  The sum in \eqref{eq:local-Fourier-limit} is then
empty, so the Fourier limit is zero.
Otherwise, $c(\boldsymbol e)$ is unchanged when the coordinates outside
$S$ are omitted.  Since the right-hand side of
\eqref{eq:local-Fourier-limit} depends only on $c(\boldsymbol e)$,
the limit is $\widehat\nu_{d,S}(\boldsymbol e|_{S})$.
Thus
\[
 \lim_{X\to\infty}\widehat\nu_{d,X,T}(\boldsymbol e)=
 \begin{cases}
 \widehat\nu_{d,S}(\boldsymbol e|_{S}),
      &e_\ell=0\text{ for all }\ell\in T\setminus S,\\
 0,   &\text{otherwise}.
 \end{cases}
\]
Applying the inversion formula and retaining only exponents supported
on $S$, we obtain
\[
 \begin{aligned}
 \lim_{X\to\infty}\nu_{d,X,T}(\{\boldsymbol z\})
 &=d^{-|T|}
   \sum_{\boldsymbol e\in\{0,\ldots,d-1\}^S}
     \widehat\nu_{d,S}(\boldsymbol e)
     \prod_{\ell\in S}z_\ell^{-e_\ell}\\
 &=d^{-|T\setminus S|}\nu_{d,S}(\{\boldsymbol z|_S\}).
 \end{aligned}
\]
This is \eqref{eq:local-pointwise}, and its right-hand side is the
point mass of $\nu_{d,S}\otimes u_d^{\otimes(T\setminus S)}$ at
$\boldsymbol z$.
\end{proof}

As $\mathfrak p$ runs over the primes above $p$, the characters
$\chi_{\mathfrak p}$ run over all characters modulo $p$ of exact
order $d$.  Indeed, the automorphism $\zeta_d\mapsto\zeta_d^v$
takes $\chi_{\mathfrak p}$ to $\chi_{\mathfrak p}^v$, and the
$\varphi(d)$ resulting characters are distinct.
Consequently, if
\[
 f((z_\ell^v)_{\ell\in T})=f((z_\ell)_{\ell\in T})
 \qquad(v\in U_d),
\]
its average over $\mathfrak p\mid p$ equals its value for any
one character of exact order $d$.

Let $(Z_{d,\ell})_{\ell\in S}$ have law $\nu_{d,S}$.
For $\ell\notin S$, let $Z_{d,\ell}$ be independent uniform
variables on $\mu_d$, independent also of the $S$-tuple.
For a rational prime $\ell$ and $z\in\mu_d$, put
\begin{equation}\label{eq:local-factor}
 F_{d,\ell}(z)=\prod_{u\in U_d}(1-z^u/\ell)^{-1},
 \qquad H_d(z)=\sum_{u\in U_d}z^u.
\end{equation}
Pairing $u$ with $-u$ shows that $F_{d,\ell}(z)>0$.
Its real logarithm satisfies
\begin{equation}\label{eq:local-log}
 \log F_{d,\ell}(z)
 =\sum_{j\ge1}\frac{H_d(z^j)}{j\ell^j}.
\end{equation}
This is an absolutely convergent real series.
Define, with primes ordered increasingly,
\begin{equation}\label{eq:Y-definition}
 Y_d=\prod_\ell F_{d,\ell}(Z_{d,\ell}).
\end{equation}

\begin{proposition}\label{prop:random-convergence}
The product in \eqref{eq:Y-definition} converges almost surely to a
random variable in $(0,\infty)$.
\end{proposition}

\begin{proof}
For $\ell\notin S$, put $G_\ell=\log F_{d,\ell}(Z_{d,\ell})$.
These variables are independent, and uniformity on $\mu_d$ gives
\[
 \E H_d(Z_{d,\ell}^j)=\varphi(d)\mathbf1_{d\mid j}.
\]
Using \eqref{eq:local-log}, we obtain
\[
 \E G_\ell=-\frac{\varphi(d)}d\log(1-\ell^{-d})
           =O_d(\ell^{-d}),\qquad
 |G_\ell|\le\frac{\varphi(d)}{\ell-1}
           \le\frac{2\varphi(d)}\ell.
\]
Hence
\[
 \sum_{\ell\notin S}|\E G_\ell|<\infty,\qquad
 \sum_{\ell\notin S}\operatorname{Var}(G_\ell)<\infty.
\]
Kolmogorov's convergence theorem implies almost sure convergence
of $\sum_{\ell\notin S}(G_\ell-\E G_\ell)$.
Adding the sum of the means and the finitely many terms above $S$
proves that $\sum_\ell\log F_{d,\ell}(Z_{d,\ell})$ is finite almost
surely.  Exponentiation gives $0<Y_d<\infty$.
\end{proof}

\Needspace{7\baselineskip}
\begin{proposition}\label{prop:atomless}
The laws of $\log Y_d$ and $Y_d$ have no atoms.
\end{proposition}

\begin{proof}
List the primes outside $S$ as $q_1<q_2<\cdots$.
Choose one representative from each pair $\{z,-z\}\subset\mu_d$,
including $1$, and call the resulting set $R$.
Write
\[
 Z_{d,q_i}=\epsilon_iV_i,\qquad V_i\in R,\quad\epsilon_i\in\{-1,1\}.
\]
For each $i$, $V_i$ is uniform on $R$ and $\epsilon_i$ is an
independent fair sign.
These pairs are independent across $i$ and independent of the
$S$-tuple.  Put
\[
 I_n=\{i\le n:V_i=1\},\qquad K_n=|I_n|.
\]
Since $\Prob(V_i=1)=2/d$, the strong law gives $K_n/n\to2/d$
almost surely.

All $u\in U_d$ are odd, so
\[
 F_{d,q_i}(1)=(1-q_i^{-1})^{-\varphi(d)},\qquad
 F_{d,q_i}(-1)=(1+q_i^{-1})^{-\varphi(d)}.
\]
Thus
\[
 \alpha_i=\log F_{d,q_i}(1)-\log F_{d,q_i}(-1)
          =\varphi(d)\log\frac{1+q_i^{-1}}{1-q_i^{-1}}>0.
\]
For fixed $n$, condition on
\[
 \mathcal H_n=\sigma\!\left(
 (V_i,\epsilon_i\mathbf1_{V_i\ne1})_{i\le n},
 (Z_{d,q_i})_{i>n},(Z_{d,\ell})_{\ell\in S}
 \right).
\]
By \cref{prop:random-convergence}, the logarithmic series converges
almost surely.  Separating the unfixed signs gives
\[
 \log Y_d=B_n+\sum_{i\in I_n}\alpha_i\mathbf1_{\epsilon_i=1},
\]
where $B_n$ is finite and $\mathcal H_n$-measurable.
Conditionally on $\mathcal H_n$, the remaining signs are independent
and fair.

For any fixed real $x$, subsets of $I_n$ giving the sum $x-B_n$
form an antichain, since every $\alpha_i$ is positive.
An antichain of subsets of an $r$-element set has size at most
$\binom r{\lfloor r/2\rfloor}$.
Consequently,
\[
 \Prob(\log Y_d=x\mid\mathcal H_n)
 \le2^{-K_n}\binom{K_n}{\lfloor K_n/2\rfloor}
 \ll(K_n+1)^{-1/2}.
\]
The conditional probability is bounded by $1$ and tends to zero
almost surely because $K_n\to\infty$.
Taking expectations and using dominated convergence yields
$\Prob(\log Y_d=x)=0$.
\end{proof}

\section{The limiting law and class groups}\label{sec:weak}

Fix an even integer $d\ge4$, and put
\[
 n=\varphi(d),\qquad m=n/2,\qquad K=\Q(\zeta_d),\qquad
 S=\{\ell\text{ prime}:\ell\mid d\}.
\]
In this section we abbreviate
\[
 F_\ell=F_{d,\ell},\quad H=H_d,\quad Z_\ell=Z_{d,\ell},\quad Y=Y_d.
\]
For $p\in\Pcal_d$, choose a character $\psi_p$ modulo $p$ of exact
order $d$, and set
\[
 W_p=\prod_{a\in U_d}L(1,\psi_p^a)>0.
\]
The product does not depend on the choice of $\psi_p$: replacing
$\psi_p$ by $\psi_p^v$, $v\in U_d$, permutes its factors.
For functions of $p\in\Pcal_d$ and sets $E\subseteq\Pcal_d$, write
\[
 \E_X f=\frac1{\pi_d(X)}\sum_{\substack{p\le X\\p\in\Pcal_d}}f(p),
 \qquad \Prob_X(E)=\E_X\mathbf1_E.
\]

\begin{theorem}\label{thm:weak-law}
For every bounded continuous $f:\R\to\R$,
\begin{equation}\label{eq:weak-law}
 \lim_{X\to\infty}\E_X f(W_p)=\E f(Y).
\end{equation}
\end{theorem}

Put $y=(\log X)^{64}$.
Granville--Soundararajan \cite[Proposition~2.2]{GS}, with $Q=X$
and $A=4$, gives
\begin{equation}\label{eq:GS-approximation}
 L(1,\chi)=\prod_{\ell\le y}(1-\chi(\ell)/\ell)^{-1}
          \bigl(1+O((\log X)^{-3})\bigr)
\end{equation}
for all but at most $X^{1/2}$ primitive characters of conductor at
most $X$.  This statement includes complex characters.
The error follows from
$A^2\log X/y^{1/(4A)}=16(\log X)^{-3}$; the conditions
$A\le\log\log X$ and $y\le X/2$ hold for sufficiently large $X$.

Let $\mathcal E_X$ consist of the primes $p\le X$ in $\Pcal_d$
for which \eqref{eq:GS-approximation} fails for some $\psi_p^a$,
$a\in U_d$.  Characters of different prime conductors are distinct,
so
\[
 |\mathcal E_X|\le X^{1/2},\qquad
 \Prob_X(\mathcal E_X)=o(1).
\]
Extend $F_\ell$ and $H$ to zero by $F_\ell(0)=1$ and $H(0)=0$.
Thus the Euler factor at $\ell=p$ is included without exception.
Multiplying \eqref{eq:GS-approximation} over $a\in U_d$ and pairing
conjugates gives, uniformly for $p\notin\mathcal E_X$,
\begin{equation}\label{eq:log-approximation}
 \log W_p=\sum_{\ell\le y}\log F_\ell(\psi_p(\ell))
                 +O_d((\log X)^{-3}).
\end{equation}

We also use the following form of the Hecke Siegel--Walfisz theorem:
for fixed $B,D>0$,
\begin{equation}\label{eq:HSW}
 \sum_{\Nm\mathfrak p\le X}\vartheta(\mathfrak p)
 \ll_{d,B,D}\frac{X}{(\log X)^D}
 \quad\left(
 \vartheta\ne1,\ \Nm\mathfrak f(\vartheta)\le(\log X)^B
 \right).
\end{equation}
Here $\vartheta$ is a finite-order Hecke character of $K$, taken
primitive and extended by zero at its conductor.
The estimate is uniform in $\vartheta$; see \cite{Goldstein} and
\cite[Lemma~33]{BabuEtAl}.  In the latter formulation, taking
$\epsilon=1/(4B)$ gives an exponential saving in
$(\log X)^{1/4}$, which implies \eqref{eq:HSW}.
The implied constant may be ineffective; real characters with a
possible exceptional zero are included.

\begin{proposition}\label{prop:empirical-tail}
For every fixed $R>\max S$,
\begin{equation}\label{eq:tail-L2}
 \lim_{X\to\infty}\E_X
 \left|\sum_{R<\ell\le y}\frac{H(\psi_p(\ell))}{\ell}\right|^2
 =n\sum_{\ell>R}\ell^{-2}.
\end{equation}
Consequently, for every $\eps>0$,
\begin{equation}\label{eq:log-tail-probability}
 \lim_{R\to\infty}\limsup_{X\to\infty}
 \Prob_X\!\left(
 \left|\sum_{R<\ell\le y}\log F_\ell(\psi_p(\ell))\right|>\eps
 \right)=0.
\end{equation}
\end{proposition}

\begin{proof}
Write $M(y)=\sum_{\ell\le y}\ell^{-1}\le\log y$.
Since $|H|\le n$, the contribution of $p\le y$ to the left side
of \eqref{eq:tail-L2} is
\[
 O_d\!\left(\frac{yM(y)^2}{\pi_d(X)}\right)=o(1).
\]
We may therefore assume $p>y$.

For $\ell,q\le y$, expanding $H$ gives
\begin{equation}\label{eq:tail-orientation}
 \begin{aligned}
 H(\psi_p(\ell))\overline{H(\psi_p(q))}
 &=\sum_{a,b\in U_d}\psi_p(\ell)^a\psi_p(q)^{-b}\\
 &=\frac1n\sum_{\mathfrak p\mid p}
       \sum_{a,b\in U_d}\eta_{d,\ell^a q^{-b}}(\mathfrak p).
 \end{aligned}
\end{equation}
The second equality holds for the full sum over $(a,b)$:
each character $\chi_{\mathfrak p}$ is a faithful power of $\psi_p$,
and taking such a power permutes $U_d^2$.

For $c\in\Q^\times$ define
\[
 C_X(c)=\frac1{n\pi_d(X)}
       \sum_{\substack{y<p\le X\\p\in\Pcal_d}}
       \sum_{\mathfrak p\mid p}\eta_{d,c}(\mathfrak p).
\]
We estimate $C_X(\ell^a q^{-b})$ uniformly for
$R<\ell,q\le y$ and $a,b\in U_d$.
If $\ell=q$ and $a=b$, then $c=1$ and
\[
 C_X(1)=1+O_d\!\left(\frac y{\pi_d(X)}\right).
\]
In every other case, $v_r(c)\not\equiv0\pmod d$ for at least
one $r\in\{\ell,q\}$: for $\ell\ne q$ the valuations are $a$ and
$-b$, while for $\ell=q$ the valuation is $a-b\ne0\pmod d$.
Thus \eqref{eq:principal-support} implies
\[
 \eta_{d,c}\rho_\xi\ne1
 \qquad\text{for every character }\xi\pmod{2d}.
\]

Put $C_d=(d^2\rad(d))^n$.
For $\mathfrak l\mid d$,
\[
 a_{\mathfrak l}(\rho_\xi)\le v_{\mathfrak l}(2d)
 \le2v_{\mathfrak l}(d).
\]
The conductor exponent of a product is at most the maximum of
the two exponents.  Hence the local bounds in
\cref{lem:conductor} give
\begin{equation}\label{eq:twisted-conductor}
 \Nm\mathfrak f(\eta_{d,c}\rho_\xi)
 \le C_d
 \prod_{\substack{r\notin S\\d\nmid v_r(c)}}r^n.
\end{equation}
For $c=\ell^a q^{-b}$ this is at most
$C_d(\ell q)^n\le C_d(\log X)^{128n}$.
Choose $B=128n+1$.  For sufficiently large $X$, all these conductors
are at most $(\log X)^B$.

Inserting \eqref{eq:selector} and passing to all prime ideals,
as in the proof of \cref{prop:local-law}, gives
\[
 C_X(c)=\frac1{n\varphi(2d)\pi_d(X)}
   \sum_{\xi\,(\mathrm{mod}\,2d)}\overline{\xi(d+1)}
   \sum_{\Nm\mathfrak p\le X}(\eta_{d,c}\rho_\xi)(\mathfrak p)
   +O_d\!\left(\frac{X^{1/2}+y}{\pi_d(X)}\right).
\]
The error counts prime ideals of residue degree at least two and
those above $p\le y$.  The conductors are supported above
$S\cup\{\ell,q\}$, so these omissions also account for every
prime where the primitive product differs from the original
character product.
Apply \eqref{eq:HSW} with $D=6$ and use
$\pi_d(X)\asymp_d X/\log X$.  Uniformly over the nondiagonal terms,
\[
 |C_X(c)|\ll_d\Delta_X,\qquad
 \Delta_X=(\log X)^{-5}+\frac{X^{1/2}+y}{\pi_d(X)}.
\]
After \eqref{eq:tail-orientation}, the sum of the absolute weights
of these terms is at most $n^2M(y)^2$.  Therefore
\[
 \begin{aligned}
 \E_X\left|\sum_{R<\ell\le y}
                  \frac{H(\psi_p(\ell))}{\ell}\right|^2
 &=n\sum_{R<\ell\le y}\ell^{-2}
       +O_d\!\left(M(y)^2\Delta_X\right)\\
 &=n\sum_{R<\ell\le y}\ell^{-2}+o(1).
 \end{aligned}
\]
Since $y\to\infty$, this proves \eqref{eq:tail-L2}.

For $z\in\mu_d\cup\{0\}$, \eqref{eq:local-log} gives
\[
 \left|\log F_\ell(z)-\frac{H(z)}\ell\right|
 \le n\sum_{j\ge2}\frac1{j\ell^j}
 \le\frac{n}{2\ell(\ell-1)}
 \le\frac n{\ell^2}.
\]
If $R$ is large enough that $n\sum_{\ell>R}\ell^{-2}\le\eps/2$,
Chebyshev's inequality yields
\[
 \limsup_{X\to\infty}
 \Prob_X\!\left(
 \left|\sum_{R<\ell\le y}\log F_\ell(\psi_p(\ell))\right|>\eps
 \right)
 \le\frac{4n}{\eps^2}\sum_{\ell>R}\ell^{-2}.
\]
Letting $R\to\infty$ proves \eqref{eq:log-tail-probability}.
\end{proof}

\begin{proof}[Proof of \cref{thm:weak-law}]
Put
\[
 \begin{aligned}
 L_p&=\log W_p,& L_p^{(R)}
   &=\sum_{\ell\le R}\log F_\ell(\psi_p(\ell)),\\
 L&=\log Y,& L^{(R)}
   &=\sum_{\ell\le R}\log F_\ell(Z_\ell).
 \end{aligned}
\]
For fixed $R>\max S$, take $T=\{\ell:\ell\le R\}$ in
\cref{prop:local-law}.  Since
\[
 F_\ell(z^v)=F_\ell(z)\qquad(v\in U_d),
\]
the sum $L_p^{(R)}$ has the same value for every choice of the
prime ideal above $p$.  Omitting the finitely many $p\le R$,
the local law therefore gives
\[
 \lim_{X\to\infty}\E_X f(L_p^{(R)})=\E f(L^{(R)})
\]
for every bounded continuous $f$.
Also, $L^{(R)}\to L$ almost surely by \cref{prop:random-convergence}.

By \eqref{eq:log-approximation} and \eqref{eq:log-tail-probability},
including the exceptional set $\mathcal E_X$ of empirical mass $o(1)$,
\begin{equation}\label{eq:truncation-in-probability}
 \lim_{R\to\infty}\limsup_{X\to\infty}
 \Prob_X(|L_p-L_p^{(R)}|>\eps)=0
 \qquad(\eps>0).
\end{equation}
For a bounded $1$-Lipschitz function $f$,
\[
 \begin{aligned}
 |\E_X f(L_p)-\E f(L)|
 &\le\eps+2\|f\|_\infty
       \Prob_X(|L_p-L_p^{(R)}|>\eps)\\
 &\quad+|\E_X f(L_p^{(R)})-\E f(L^{(R)})|
       +|\E f(L^{(R)})-\E f(L)|.
 \end{aligned}
\]
Let $X\to\infty$, then $R\to\infty$, and finally $\eps\downarrow0$.
The second line tends to zero by the fixed-cutoff limit and bounded
convergence.  Equation \eqref{eq:truncation-in-probability} handles
the first line.  Thus $L_p$ converges in empirical distribution to $L$.
Bounded Lipschitz functions determine weak convergence, and applying
the continuous map $x\mapsto e^x$ proves \eqref{eq:weak-law}.
\end{proof}

We now restrict to $p\in\Pcal_d$ with $p>d+1$, so that $N=(p-1)/d>1$.
For a nontrivial character $\chi$ modulo $p$, put
\[
 B_{1,\chi}=\frac1p\sum_{b=1}^{p-1}b\chi(b),\qquad
 G_p=\Gcal_{d,p}
   =\Nm_{K/\Q}(B_{1,\psi_p^{-1}})
   =\prod_{a\in U_d}B_{1,\psi_p^{-a}}.
\]

\begin{lemma}\label{lem:Bernoulli-norm}
For $p\in\Pcal_d$, $p>d+1$, one has $G_p\in\Z_{>0}$ and
\begin{equation}\label{prop2.1}
 G_p=\frac{p^m}{\pi^n}W_p.
\end{equation}
\end{lemma}

\begin{proof}
Let $\chi$ have exact order $d$.
Its kernel has order $N>1$, so choose $2\le a\le p-1$
with $\chi(a)=1$.  Permuting the nonzero residues by multiplication
by $a$ gives
\[
 (a-1)B_{1,\chi}
 =\sum_{b=1}^{p-1}\left\lfloor\frac{ab}p\right\rfloor\chi(b)
 \in\Z[\zeta_d].
\]
Also $pB_{1,\chi}\in\Z[\zeta_d]$.  Since $\gcd(a-1,p)=1$,
B\'ezout's identity implies $B_{1,\chi}\in\Z[\zeta_d]$.
Its norm is therefore an integer.

All characters $\psi_p^a$ are odd and primitive.  Their functional
equations \cite[Theorem~4.2]{Washington} imply
\[
 B_{1,\chi}B_{1,\bar\chi}
   =\frac p{\pi^2}L(1,\chi)L(1,\bar\chi)>0.
\]
Multiplication over the $m$ conjugate pairs proves both
$G_p>0$ and \eqref{prop2.1}.
\end{proof}

Define
\begin{equation}\label{eq:blocks}
 B_p^-=\bigoplus_{a\in U_d}A_{aN},\qquad
 B_p^+=\bigoplus_{a\in U_d}A_{p-aN}.
\end{equation}

\begin{proposition}\label{prop:valuation}
For $p\in\Pcal_d$ with $p>d+1$,
\begin{equation}\label{eq:exact-valuation}
 |B_p^-|=p^{v_p(G_p)}.
\end{equation}
\end{proposition}

\begin{proof}
Write $d=2^s t$ with $t$ odd, and put $L=\Q(\zeta_p)$.
For $e=2^s u$, $u\mid t$, let $L_e$ be the degree-$e$ subfield
of $L$ and put $H_e=\Gal(L/L_e)$.
Since $[L:L_e]=(p-1)/e$ is odd, $L_e$ is a CM field.
For a number field $M$, write $A(M)$ for its $p$-primary class group;
in particular, $A(L)=A$.

Let $j_e$ and $N_e$ be extension and norm on class groups.
They satisfy
\[
 N_ej_e=[L:L_e],\qquad
 j_eN_e=\sum_{\sigma\in H_e}\sigma.
\]
Multiplication by $[L:L_e]$ is invertible on $p$-groups.
The first identity makes $j_e$ injective; the second makes its image
equal to $A^{H_e}$.  Hence
\begin{equation}\label{eq:descent}
 A(L_e)\simeq A^{H_e}
 =\bigoplus_{j=0}^{e-1}A_{j(p-1)/e}.
\end{equation}
The same argument for $L_e/L_e^+$ identifies $A(L_e^+)$
with the even part.  Thus
\begin{equation}\label{eq:relative-valuation}
 v_p(h^-(L_e))
 =\sum_{\substack{1\le j<e\\j\text{ odd}}}
      v_p(|A_{j(p-1)/e}|),
 \qquad h^-(L_e)=\frac{h(L_e)}{h(L_e^+)}.
\end{equation}

For $f=2^s v$, $v\mid t$, set
\[
 b_f=\sum_{a\in U_f}v_p(|A_{a(p-1)/f}|),\qquad
 g_f=v_p\!\left(\prod_{\ord(\chi)=f}B_{1,\chi}\right),
\]
where the product is over characters modulo $p$ of exact order $f$.
A character of $L_e$ is odd precisely when its order $f$ satisfies
$f\mid e$ and $e/f$ odd.  Grouping \eqref{eq:relative-valuation}
by $f$ gives
\begin{equation}\label{eq:class-sum-by-order}
 v_p(h^-(L_e))=\sum_{v\mid u}b_{2^s v}.
\end{equation}
The relative class-number formula \cite[Theorem~4.17]{Washington} is
\[
 h^-(L_e)=Q(L_e)w(L_e)
           \prod_{\chi\in X^-(L_e)}\left(-\frac12B_{1,\chi}\right).
\]
Here $Q(L_e)\in\{1,2\}$ is the Hasse unit index and $w(L_e)$ is
the number of roots of unity.  Since $e\le d<p-1$, $L_e$ is
a proper subfield of $\Q(\zeta_p)$, so $w(L_e)=2$.
All prefactors are $p$-units.  Grouping the product by character
order therefore gives
\begin{equation}\label{eq:Bernoulli-sum-by-order}
 v_p(h^-(L_e))=\sum_{v\mid u}g_{2^s v}.
\end{equation}
Comparison for all $u\mid t$, followed by M\"obius inversion,
yields $b_d=g_d$.  The definitions give
$b_d=v_p(|B_p^-|)$ and $g_d=v_p(G_p)$, proving
\eqref{eq:exact-valuation}.
\end{proof}

The reflection theorem \cite{Leopoldt,Washington} gives
\[
 r_p(A_{p-aN})\le r_p(A_{aN})\qquad(a\in U_d).
\]
Together with \cref{prop:valuation}, this gives
\begin{equation}\label{eq:block-rank-bound}
 r_p(B_p^+)\le r_p(B_p^-)\le v_p(G_p).
\end{equation}

\begin{proof}[Proof of \cref{th1.1}]
By \cref{thm:weak-law,lem:Bernoulli-norm}, the empirical law of
$G_p/p^m$ converges to the law of $Y/\pi^n$.
The latter is atomless by \cref{prop:atomless}.  Since $G_p\in\Z_{>0}$,
\[
 v_p(G_p)\ge m\quad\Longleftrightarrow\quad G_p/p^m\in\Z_{>0}.
\]
The set $\Z_{>0}$ is closed and has measure zero under the law of
$Y/\pi^n$.  The Portmanteau theorem
\cite[Theorem~2.1]{BillingsleyConvergence} therefore implies that
this event has relative density zero.  Consequently,
\[
 S_d=\{p\in\Pcal_d:p>d+1,\ v_p(G_p)\le m-1\}
\]
has relative density one.  For $p\in S_d$, \eqref{eq:exact-valuation} gives
\[
 |B_p^-|=p^{v_p(G_p)}\le p^{m-1}.
\]
The even-block rank bound then follows by reflection, as in
\eqref{eq:block-rank-bound}:
\[
 r_p(B_p^+)\le r_p(B_p^-)\le v_p(|B_p^-|)\le m-1.
\]
\end{proof}

\begin{proof}[Proof of \cref{cor1.2}]
For $d=4,6$, $m=1$, so $|B_p^-|=1$ and $r_p(B_p^+)=0$
on the density-one set in \cref{th1.1}; both blocks are trivial.
For $d=8,10,12$, $m=2$, so $|B_p^-|\le p$.
Thus $B_p^-$, and hence each of its direct summands, is trivial
or isomorphic to $\Z/p\Z$.
\end{proof}

\appendix
\section{The component \texorpdfstring{$A_{34}$}{A34}}\label{app:A34-computation}

The factorization
\begin{equation}\label{eq:B34-factorization}
 B_{34}=\frac{2577687858367}{6}
       =\frac{17\cdot151628697551}{6}
\end{equation}
reduces the proof to one prime.  Indeed, for $p\ge37$, reflection
and Herbrand's theorem \cite{Washington} give
\[
 A_{34}\ne0\ \Longrightarrow\ A_{p-34}\ne0
 \ \Longrightarrow\ p\mid\operatorname{num}(B_{34}).
\]
Thus the only possible exception is $p_0=151628697551$;
see also \cite[Section~3, Remark~1]{Q}.
All odd primes below $37$ are regular, so $A=0$ for those primes.

We now fix
\begin{equation}\label{eq:A34-parameters}
 p=p_0,\qquad q=12p+1=1819544370613,\qquad z=4096.
\end{equation}
Here $p$ and $q$ are prime, and $z$ has order $p$ modulo $q$.
Let $e(a)$ be the integer with $1\le e(a)<p$ and
$e(a)\equiv a^{-34}\pmod p$, and put
\begin{equation}\label{eq:A34-test}
 T=\left(\prod_{a=1}^{(p-1)/2}
          (z^a-z^{-a})^{e(a)}\right)^{12}\in\F_q^\times.
\end{equation}
The cyclotomic-unit criterion of \cite[Section~4]{BuhlerHarvey}
(see also \cite[Section~8.3]{Washington}) gives
\begin{equation}\label{eq:A34-criterion}
 T\ne1\quad\Longrightarrow\quad A_{34}=0.
\end{equation}
The exponent $a^{p-1-34}$ in that criterion may be reduced modulo $p$,
since $12p=q-1$.

To evaluate the product, we use doubling orbits as in
\cite[Section~6]{BuhlerHarvey}.  The relevant orders are
\begin{equation}\label{eq:A34-orbit-data}
 t=\ord_p(2)=891933515,\qquad
 \ord_p(7)=p-1=170t.
\end{equation}
Since $t$ is odd, the subgroup $\langle2,-1\rangle$ has index $85$
in $\F_p^\times$.  Consequently,
\begin{equation}\label{eq:A34-representatives}
 \{7^j2^i:0\le j<85,\ 0\le i<t\}
\end{equation}
contains exactly one representative of each pair $\{a,-a\}$.
Replacing $a$ by $-a$ changes its factor only by a sign, since
$e(-a)=e(a)$.  The final twelfth power removes this sign, so these
representatives give the same value of $T$ as \eqref{eq:A34-test}.

For a block starting at $a_0$, the function \texttt{block} initializes
\[
 x=z^{a_0},\qquad y=z^{-a_0},\qquad e=e(a_0),\qquad V=1.
\]
With $u\equiv2^{-34}\pmod p$, each iteration multiplies $V$ by
$(x-y)^e$ and then makes the updates
\begin{equation}\label{eq:A34-recurrence}
 x\leftarrow x^2,\qquad y\leftarrow y^2,\qquad
 e\leftarrow ue\bmod p.
\end{equation}
Here $x,y,V$ are computed modulo $q$, and $e$ is kept between $1$
and $p-1$.

Put $s=5000000$ and $h=\lceil t/s\rceil=179$.
Each of the $85$ orbits is divided into $h$ consecutive blocks.
For $0\le j<85$ and $0\le b<h$, the block at offset $bs$ in orbit $j$
starts at $a_0=7^j2^{bs}\pmod p$ and has $\min(s,t-bs)$ terms.
The function \texttt{job} assigns these blocks the indices $jh+b+1$.
Thus the $85h=15215$ blocks cover all
$85t=75814348775=(p-1)/2$ factors exactly once.
The accompanying PARI/GP program \texttt{verify\_A34.gp}, available at
\url{https://github.com/taozhy07/vandiver-a34-verification},
evaluates the blocks in parallel with \texttt{parvector} and combines
them with \texttt{vecprod}.  It prints \texttt{PASS} if the twelfth
power of the resulting product is not $1$ modulo $q$, and
\texttt{INCONCLUSIVE} otherwise. The complete calculation can be run from the directory
containing the file with
\begin{verbatim}
  gp -f -q verify_A34.gp
\end{verbatim}
On a 2021 MacBook Pro with an Apple M1 Pro chip and $16$\,GB of RAM,
the complete run took approximately three hours and returned
\texttt{PASS}.

\begin{proof}[Proof of \cref{thm:A34-vanishing}]
The output \texttt{PASS} gives $T\ne1$, so
\eqref{eq:A34-criterion} implies $A_{34}=0$ when $p=p_0$.
The reduction above covers every other odd prime, completing the proof.
\end{proof}

\section*{Acknowledgements}
The first author is supported by NSFC 12231009. The second author is supported by the Natural Science Foundation of Anhui Province (Grant No. 2508085QA017).

\end{document}